\documentclass{article}
\usepackage{amsbsy,amssymb,amsmath,amsfonts, diagrams, amsthm, hyperref}

\newtheorem{theorem}{Theorem}[section] 
\newtheorem{lemma}[theorem]{Lemma}     
\newtheorem{corollary}[theorem]{Corollary}
\newtheorem{proposition}[theorem]{Proposition}

\newtheorem{remark}[theorem]{Remark}
\newtheorem{definition}[theorem]{Definition}

\newcommand{\decomp}{\mathcal{D}}
\newcommand{\decompX}{\{X_i\}_{i=0}^{\infty}}
\newcommand{\decompY}{\{Y_i\}_{i=0}^{\infty}}
\newcommand{\LS}{L}
\newcommand{\RS}{R}
\newcommand{\LSo}{L_{\mathcal{D}_1}}
\newcommand{\RSo}{R_{\mathcal{D}_1}}
\newcommand{\LSt}{L_{\mathcal{D}_2}}
\newcommand{\RSt}{R_{\mathcal{D}_2}}

\newcommand{\X}{\mathcal{X}}
\newcommand{\Y}{\mathcal{Y}}
\newcommand{\Z}{\mathcal{Z}}
\newcommand{\opX}{\mathcal{L}(\X)}

\newcommand{\idI}{\mathcal{I}(X)}
\newcommand{\AD}{\mathcal{A}}
\newcommand{\TD}{T_{\mathcal{A}}}
\newcommand{\N}{\mathbb{N}}
\newcommand{\CN}{\mathbb{C}}
\newcommand{\LCN}{\lambda\in\mathbb{C}}
\newcommand{\linf}{\ell_{\infty}}
\newcommand{\oplinf}{\mathcal{L}(\ell_{\infty})}
\newcommand{\ssop}[1]{\mathcal{S}(#1)}
\newcommand{\sso}{\mathcal{S}(\ell_{\infty})}
\newcommand{\comp}[1]{\mathcal{K}(#1)}
\newcommand{\Trest}[1] {\lambda_{#1}I_{|#1} + K_{#1}}
\newcommand{\seq}[1] {\displaystyle \{#1_i\}_{i=1}^{\infty}}
\newcommand{\seqo}[1] {\displaystyle \{#1_i\}_{i=0}^{\infty}}
\newcommand{\sumspace}{\big (\sum_{i=0}^{\infty} Y_i\big )_{p}}
\newcommand{\sumY}{\big (\sum Y\big )_p}
\newcommand{\sumYI}{\big (\sum Y\big )_\infty}
\newcommand{\sumX}{\big (\sum \X\big )_p}
\newcommand{\opsumY}{\mathcal{L}\big (\big (\sum Y\big )_p\big )}

\newcommand{\les}{\sigma_{l.e.}(T)}
\newcommand{\IX}{\mathcal{M}_\X}
\newcommand{\propP}{{\textbf{P}}}

\begin{document}
\begin{center}{\Huge Commutators on $\ell_{\infty}$}
\end{center}

\begin{center}{\Large D. Dosev  \footnote{Research supported in part by NSF grant DMS-0503688}, W. B. Johnson \footnotemark[\value{footnote}] }
\end{center}


\begin{abstract}
The operators on $\linf$ which are commutators are those not of the form $\lambda I + S$ with 
$\lambda\neq 0$ and $S$ strictly singular.
\end{abstract}
\section{Introduction}\label{intro}
The commutator of two elements $A$ and $B$ in a Banach algebra is given by
$$
[A,B] = AB - BA.
$$
A natural problem that arises in the study of  derivations on a Banach algebra $\X$ is to classify the commutators in the algebra.
Using a result of Wintner(\cite{Wintner}), who proved that the identity in a unital Banach algebra is not a commutator, with no effort one can also 
show that no operator of the form  $\lambda I + K$, where $K$  belongs to a norm closed ideal $\idI$ of $\opX$ and $\lambda\neq 0$, 
is a commutator in the Banach algebra $\opX$ of all bounded linear operators on the Banach space $\X$. 
The latter fact can be easily seen just by observing that the  quotient algebra $\opX / \idI$ also satisfies  the conditions of Wintner's theorem. 

\noindent
In 1965 Brown and Pearcy (\cite{BrownPearcy}) made a breakthrough by proving that the only operators on $\ell_2$ that are not commutators are the ones of the form $\lambda I + K$, where $K$  is compact and $\lambda\neq 0$.  Their result suggests what the classification on the other classical sequence spaces might be, and, in 1972, Apostol (\cite{Apostol_lp}) proved that every non-commutator on the space $\ell_p$ for $1<p<\infty$ is of the form 
$\lambda I + K$, where $K$ is compact and $\lambda\neq 0$. One year later he proved that the same classification holds in the case of $\X = c_0$  (\cite{Apostol_c0}). Apostol proved some partial results on $\ell_1$, but only 30 year later was the same classification proved for $\X = \ell_1$ by the first author (\cite{Dosev}).
Note that if $\X = \ell_p$ ($1\leq p<\infty$) or $\X = c_0$, the ideal of compact operators $K(\X)$ is the largest proper ideal  in $\mathcal{L}(\X)$ (\cite{GM}, see also \cite[Theorem 6.2]{Whitley}). The classification of the commutators on $\ell_p$, $1\leq p<\infty$, and partial results on other spaces suggest the following 

\noindent
\textbf{Conjecture.}
Let $\X$ be a Banach space such that $\X\simeq \sumX$, $1\leq p\leq\infty$ or $p=0$
(we say that such a space admits a Pe\l czy\'nski decomposition). Assume that $\opX$ has a largest ideal
$\mathcal{M}$. Then every non-commutator on $\X$ has the form $\lambda I + K$, where $K\in \mathcal{M}$  and $\lambda\neq 0$.

\noindent
In \cite{Apostol_lp} Apostol obtained a partial result regarding the commutators on $\linf$. He proved that if  $T\in\oplinf$  and there exists a sequence of projections $(P_n)_{n=1}^{\infty}$ on $\linf$ such that $P_n(\linf)\simeq\linf$ for $n = 1,2,\ldots$ and $\|P_nT\|\to 0$ as $n\to\infty$, then $T$ is a commutator. This condition is clearly satisfied if $T$ is a compact operator, but, as the first author showed in \cite{Dosev}, it is also satisfied if $T$ is strictly singular, which is an essential step for proving the conjecture for $\linf$.

\noindent
In order to give a positive answer to the conjecture one has to prove

\begin{itemize}
\item Every operator $T\in\mathcal{M}$ is a commutator
\item If $T\in\opX$ is not of the form $\lambda I + K$, where $K\in\mathcal{M}$ and $\lambda\neq 0$, then $T$ is
a commutator.
\end{itemize}
In this paper we will give positive answer to this conjecture for the space $\linf$.
\section{Notation and basic results} 
For a Banach space $\X$ denote by the $\opX$, $\comp{\X}$, $\mathcal{C}(\X)$ and  $S_{\X}$ the space of all bounded linear operators, the ideal of compact operators, the set of all finite co-dimensional subspaces of $\X$ and the unit sphere of $\X$. By {\it ideal} we always mean closed, non-zero, proper ideal.
A map from a Banach space $\X$ to a Banach space $\Y$ is said to be strictly singular if whenever
the restriction of $T$ to a subspace $M$ of $X$ has a continuous inverse, $M$ is finite
dimensional. In the case where $\X\equiv\Y$, the set of strictly singular operators forms an ideal which we will denote by $\ssop{\X}$. Recall that for $\X = \ell_p$, $1\leq p < \infty$, $\ssop{\X} = \comp{X}$ (\cite{GM}) and on $\linf$ the ideals of strictly singular and weakly compact operators coincide (\cite[Theorem 5.5.1]{Kalton}).
 A Banach space $\X$ is called {\it{prime}} if each infinite-dimensional complemented subspace of $\X$ is isomorphic to $\X$. The spaces $\ell_p$, $1\leq p\leq\infty$, are all prime (cf. \cite[Theorem 2.a.3 and Theorem 2.a.7]{LT}). For any two subspaces (possibly not closed) $\X$ and $\Y$  of a Banach space $\Z$ let
$$
d(\X,\Y) = \inf \{\|x-y\| : x\in S_{\X},\, y\in \Y\}.
$$ 
A well known consequence of the open mapping theorem is that for any two closed subspaces $\X$ and $\Y$ of $\Z$, if $\X\cap\Y = \{0\}$ then $\X+\Y$ is a closed subspace of $\Z$ if and only if $d(\X,\Y)>0$. 
Note also that $d(\X,\Y) = 0$ if and only if $d(\Y,\X) = 0$. First we  prove a proposition that will later allow us to consider translations of an operator $T$ by a multiple of the identity instead of the operator $T$ itself. 

\begin{proposition}\label{distanceprop}
Let $\X$ be a Banach space and $T\in\opX$ be  such that there exists a 
subspace $Y\subset\X$ for which $T$ is an isomorphism on $Y$ and $d(Y,TY)>0$. Then for every $\LCN$, $(T-\lambda I)_{|Y}$ is an isomorphism and $d(Y,(T-\lambda I)Y)>0$.
\end{proposition}
\begin{proof}
First, note that the two hypotheses on Y (that $T$ is an isomorphism on $Y$ and $d(Y,TY)>0$) are together equivalent to the existence of a constant $c>0$ s.t. for all $y \in S_Y$, $d(Ty,Y)>c$. To see this, let us first assume that the hypotheses of the theorem are satisfied. Then there exists a constant $C$ such that
$\|Ty\|\geq C$ for every $y\in S_Y$. For an arbitrary $y\in S_Y$, let $z_y = \frac{Ty}{\|Ty\|}$ and then we clearly have
$$
d(Ty,Y) = \|Ty\|d(z_y, Y)\geq Cd(TY,Y) =: c > 0.
$$
To show the other direction note that for $y\in S_Y$, $0<c < d(Ty, Y) = \|Ty\|d(z_y, Y)\leq \|T\|d(z_y, Y)$.
Taking the infimum over all $z_y\in S_Y$ in the last inequality, we obtain that $d(TY, Y)>0$ and hence $d(Y, TY)>0$.
On the other hand, for all $y\in S_Y$ we have 
$$
0<c < d(Ty, Y)\leq \|Ty - \frac{c}{2}y\|\leq  \|Ty\| + \frac{c}{2},
$$
hence $\|Ty\|\geq \frac{c}{2}$, which in turn implies that $T$ is an isomorphism on $Y$.

Now it is easy to finish the proof. The condition $d(Ty,Y)>c$ for all $y\in S_Y$ is clearly satisfied if we substitute $T$ with $T-\lambda I$ since for a fixed $y\in S_Y$,
$$
d((T-\lambda I)y,Y) = \inf_{z\in Y} \|(T-\lambda I)y-z\| = \inf_{z\in Y} \|Ty-z\|= d(Ty, Y),
$$
hence $(T-\lambda I)_{|Y}$ is an isomorphism and $d(Y,(T-\lambda I)Y)>0$.
\end{proof}
\noindent
Note the following two simple facts:
\nobreak
\begin{itemize}
\item If $T\colon\X\to\X$ is a commutator on $\X$ and $S\colon\X\to\Y$ is an onto isomorphism, then  $STS^{-1}$ is a commutator on $\Y$.
\item Let $T\colon\X\to\X$ be such that  there exists $X_1\subset \X$ for which $T_{|X_1}$ is an isomorphism and $d(X_1,TX_1)>0$. If $S\colon\X\to\Y$ is an onto isomorphism, then there exists $Y_1\subset \Y$, $Y_1\simeq X_1$, such that ${STS^{-1}}_{|Y_1}$ is an isomorphism and 
$d(Y_1, STS^{-1}Y_1)>0$ (in fact $Y_1 = SX_1$). Note also that if $X_1$ is complemented in 
$\X$, then $Y_1$ is complemented in $\Y$.
\end{itemize}
Using the two facts above, sometimes we will replace an operator $T$ by an operator $T_1$ which is  similar to $T$ and possibly acts on another Banach space.

If $\seqo{Y}$ is a sequence of arbitrary Banach spaces, by $\sumspace$ we denote the space of all sequences $\seqo{y}$ where $y_i\in Y_i$, $i=0,1,\ldots$, such that $(\|y_i\|_{Y_i})\in\ell_p$ with the norm $\|(y_i)\| = \|\|y_i\|_{Y_i}\|_p$ (if $Y_i \equiv Y$ for every $i=0,1,\ldots$ we will use the notation  $\sumY$). We will only consider the case where all the spaces $Y_i$, $i=0,1\ldots$,
are uniformly  isomorphic to a Banach space $Y$, that is, there exists a constant $\lambda >0$ and sequence of 
isomorphisms $\{T_i\colon Y_i\to Y\}_{i=0}^{\infty}$ such that $\|T^{-1}\| = 1$ and $\|T\|\leq\lambda$. In this case we define an isomorphism $U\colon\sumspace\to\sumY$  via $(T_i)$ by 
\begin{equation}\label{defU}
U(y_0, y_1, \ldots) = (T_0(y_0), T_1(y_1), \ldots),
\end{equation}
and it is easy to see that $\|U\|\leq\lambda$ and $\|U^{-1}\| = 1$. Sometimes we will identify the space $\sumspace$ with  $(\sum Y)_p$ via the isomorphism $U$ when there is no ambiguity how the properties of an operator $T$ on $\sumspace$ translate to the properties of the operator $UTU^{-1}$ on $\sumY$.\\
For $y=(y_i)\in (\sum Y)_p$ , $y_i\in Y$, define the following two operators :
$$
\RS (y) = (0,y_0,y_1, \ldots)\qquad ,\qquad \LS(y) = (y_1, y_2, \ldots).
$$
The operators $\LS$ and $\RS$ are, respectively, the left and the right shift on the space $(\sum Y)_p$.
Denote by $P_i$, $i=0,1,\ldots$, the natural, norm one, projection from $\sumY$ onto  the  $i$-th component of $\sumY$, which we denote by $Y^i$. We should note that if $Y\simeq\sumY$, then some of the results in this paper are similar to results in \cite{Dosev}, but initially we do not require this condition, and, in particular, some of the results we prove here have applications to spaces like $\big (\sum \ell_q\big )_p$ for arbitrary $1\leq p,q \leq\infty$.
Our first proposition shows some basic properties of the left and the right shift as well 
as the fact that all the powers of $L$ and $R$ are uniformly bounded, which will play an important role in the sequel.
Since the proof follows immediately from the definitions we will omit it.
\begin{proposition}\label{propLRbounds1}
Consider the Banach space $\sumY$. We have the following identities
\begin{equation}\label{eq:LRbound1}
\|L^n\| = 1  \,\,,\,\, \|R^n\| =  1 \,\,\,\textrm{for every}\,\, n = 1,2,\ldots
\end{equation}
\begin{equation}\label{eq:LRidentities1}
\LS P_0 = P_0\RS = 0 \,\,,\,\, \LS\RS = I  \,\,,\,\,\RS\LS = I-P_0\,\,,\,\, \RS P_i = P_{i+1}\RS \,\,\, , \,\,\, P_i\LS = \LS P_{i+1} \,\,\, \textrm{for}\,\,\, i\geq 0 .
\end{equation} 
\end{proposition}
\noindent
Note that we can define a left and right shift on $\sumspace$ by 
$\tilde L = U^{-1}LU$ and $\tilde R = U^{-1}RU$, and, using the above proposition, we immediately have 
$\|\tilde R^n\|\leq\lambda$ and $\|\tilde L^n\|\leq\lambda$. If there is no ambiguity, we will denote the left and the right shift on  $\sumspace$ simply by $L$ and $R$.

\noindent
Following the ideas in \cite{Apostol_lp}, for $1\leq p<\infty$ and $p=0$ define the set
\begin{equation}\label{eq:ADDef}
\AD = \{T\in \sumY : \sum_{n=0}^{\infty} \RS^nT\LS^n \,\,\,\textrm{is strongly convergent}\},
\end{equation}
and for $T\in\AD$  define
$$
\TD = \sum_{n=0}^{\infty} \RS^nT\LS^n .
$$
Now using the fact that an operator $T$ is a commutator if and only if $T$ is in the range of $D_S$ for some $S$,
where $D_S$ is the inner derivation determined by $S$, defined by $D_S(T) = ST - TS$,
it is easy to see (\cite[Lemma 3]{Dosev}) that if $T\in\AD$ then
\begin{equation}\label{eq:TinADRepresentation}
T = D_{\LS}(\RS\TD) = -D_{\RS}(\TD\LS),
\end{equation}
hence $T$ is a commutator. 
\section{Commutators on $\sumY$}
The ideas in this section are similar to the ideas in \cite{Dosev}, but here we present them from a different point of view, in a more general setting and we also include the case $p=\infty$.
The following lemma is a generalization of \cite[Lemma 2.8]{Apostol_lp} in the case $p=\infty$ and 
\cite[Corollary 7]{Dosev} in the case $1\leq p<\infty$ and $p=0$. The proof presented here follows the ideas of the proof in \cite{Dosev}. Of course, some of the ideas can be traced back to the classic paper
of Brown and Pearcy (\cite{BrownPearcy}) and to Apostol's papers \cite{Apostol_lp}, \cite{Apostol_c0}, and the references therein. 
\begin{lemma}\label{wstarconvlemma}
Let $T\in\opsumY$. Then the operators $P_0T$ and $TP_0$ are commutators.
\end{lemma}
\begin{proof}
The proof shows that $P_0T$ is in the range of $D_L$ and $TP_0$ is in the range of $D_R$.
We will consider two cases depending on $p$.

\noindent
{\textbf{Case I : $p=\infty$}}\\
In this case we first observe that the series
$$
S_0 = \sum_{n=0}^{\infty} R^nP_0TL^n
$$
is pointwise convergent coordinatewise.
Indeed, let $x\in\sumYI$ and define $y_n = R^nP_0TL^nx$ for $n = 0,1,\ldots$. Note that from
the definition we immediately have $y_n\in Y^n$ so the sum $\sum_{n=0}^{\infty} y_n$ converges in the product topology on
$\sumYI$ to a point in $\sumYI$
 since  $\|y_n\|\leq \|R^n\|\|P_0\|\|T\|\|L^n\|\|x\| \leq \|T\|\|x\|$.

Secondly, we observe that $S_0$ and $L$ commute.  Because $L$ and $R$ are continuous operators on
$\sumYI$ with the product topology and $LR=I$, we have
\begin{equation}\label{eq:wstarconv1}
\begin{split}
S_0Lx&=\sum_{n=0}^{\infty}R^nP_0TL^{n+1}x =
L\left (\sum_{n=1}^{\infty}R^nP_0TL^nx \right )
= L\left (\sum_{n=0}^{\infty}R^nP_0TL^nx\right )  - LP_0Tx\\
&= LS_0x - 0
\end{split}
\end{equation}
since $LP_0=0$. That is,  $D_LS_0 =0$, as desired.

On the other hand, again using $LP_0=0$,
\begin{equation}\label{eq:wstarconv3}
\begin{split}
(I-RL)S_0x&=\sum_{n=0}^{\infty}(I-RL)R^nP_0TL^nx  
= (I-RL)P_0Tx + \underbrace{\sum_{n=1}^{\infty}(I-RL)R^nP_0TL^nx}_{0} \\
&=(I-RL)P_0Tx = P_0Tx.
\end{split}
\end{equation}
Therefore
\begin{equation}
D_L(RS_0) = (D_LR)S_0 + R(D_LS_0) = (I-RL)S_0 + 0 = P_0 S_0 = P_0T.
\end{equation}
The proof of the statement that $TP_0$ is a commutator
involves a similar modification of the proof of \cite[Lemma 2.8]{Apostol_lp}.
Again, consider the series
$$
S = \sum_{n=0}^{\infty} R^nP_0TP_0L^n .
$$
This is pointwise convergent coordinatewise and $SL=LS$ (from the above reasoning applied to the operator $TP_0$), and
\begin{equation*}
\begin{split}
D_R(-SL) &= -D_R(LS)= -RLS+LSR=-(I-P_0)S+LSR\\
&=-S+P_0S+SLR= -S+P_0S+S=P_0TP_0.
\end{split}
\end{equation*}
Now it is easy to see that
$$
D_R(LTP_0 - SL) =  RLTP_0 - \underbrace{LTP_0R}_{0} + P_0TP_0 = (I-P_0)TP_0 + P_0TP_0 = TP_0.
$$

\noindent
{\textbf{Case II :}} $1\leq p<\infty$ or $p=0$\\
In this case the proof is similar to the proof of \cite[Lemma 6 and Corollary 7]{Dosev} and we include it for completeness.
Let us consider the case $p\geq 1$ first. For any $y\in\sumY$  we have
\begin{eqnarray*}
\|\sum_{n=m}^{m+r} \RS^{n}P_iTP_j\LS^ny\|^p &=&  \|\sum_{n=m}^{m+r} \RS^{n}P_iTP_j\LS^nP_{j+n}y\|^p  
= \sum_{n=m}^{m+r} \|\RS^{n}P_iTP_j\LS^nP_{j+n}y\|^p\\
&\leq& \|P_iTP_j\|^p\sum_{n=m}^{m+r}\|P_{j+n}y\|^p 
\leq  \|P_iTP_j\|^p\sum_{n=m}^{\infty}\|P_{j+n}y\|^p.
\end{eqnarray*}
Since $\displaystyle \sum_{n=m}^{\infty}\|P_{j+n}y\|^p\to 0$ as $m\to\infty$ we have that 
$\displaystyle \sum_{n=0}^{\infty} \RS^{n}P_iTP_j\LS^n$ is strongly convergent and $P_iTP_j\in\AD$.\\
For $p=0$ a similar calculation shows 
\begin{eqnarray*}
\|\sum_{n=m}^{m+r} \RS^{n}P_iTP_j\LS^ny\| &=&  \|\sum_{n=m}^{m+r} \RS^{n}P_iTP_j\LS^nP_{j+n}y\|  
= \max_{m\leq n\leq m+r} \|\RS^{n}P_iTP_j\LS^nP_{j+n}y\|\\
&\leq&  \|P_iTP_j\|\max_{m\leq n\leq m+r}\|P_{j+n}y\|  
\end{eqnarray*}
and since $\displaystyle \max_{m\leq n\leq m+r}\|P_{j+n}y\|\to 0$ as $m\to\infty$  we apply the same argument as in the case $p\geq 1$  to obtain $P_iTP_j\in\AD$.\\
Using $P_iTP_j\in\AD$ for $i=j=0$ and (\ref{eq:TinADRepresentation}) we have 
$P_0TP_0 = D_{\LS}(\RS(P_0TP_0)_{\AD}) =$\\ $ -D_{\RS}((P_0TP_0)_{\AD}\LS)$. Again, as in \cite[Corollary 7]{Dosev}, via direct computation we obtain
\begin{eqnarray}
TP_0 &=& D_{\RS}(\LS TP_0 - (P_0TP_0)_{\AD}\LS)\\
P_0T &=& D_{\LS}(-P_0T\RS + \RS(P_0TP_0)_{\AD}).
\end{eqnarray}
\end{proof}

\noindent
Now we switch our attention to Banach spaces which in addition satisfy 
$\X\simeq\sumX$ for some $1\leq p\leq\infty$ or $p=0$. Note that the Banach space $\sumY$ satisfies this condition regardless of the space $Y$, hence we will be able to use  the results we proved so far in this section. We begin with a definition.
\begin{definition}
Let $\X$ be a Banach space such that $\X\simeq \sumX$, 
$1\leq p\leq\infty$ or $p=0$. We say that $\displaystyle \decomp =\{X_i\}_{i=0}^{\infty} $ is a decomposition of $\X$ if it forms an $\ell_p$ or $c_0$ decomposition of $\X$ into subspaces which are uniformly isomorphic to $\X$;
that is, if the following three conditions are satisfied:
\begin{itemize}
\item There are uniformly bounded projections $P_i$ on $\X$ with $P_i\X = X_i$ and $P_iP_j = 0$ for $i\neq j$
\item There exists a collection of isomorphisms $\psi_i : X_i\to \X$, $i\in\N$, such that  $\|\psi_i^{-1}\| = 1$ and
$\displaystyle \lambda = \sup_{i\in\N}\|\psi_i\| < \infty$
\item The formula $Sx = (\psi_iP_ix)$ defines a surjective isomorphism from $\X$ onto $\sumX$
\end{itemize}
\end{definition}

If $\decomp =\{X_i\}_{i=0}^{\infty}$ is a decomposition of $\X$ we have 
$\X\simeq \sumX\simeq \left ( \sum_{i=0}^{\infty} X_i\right)_{p}$, where the second isomorphic relation is via the isomorphism $U$ defined in (\ref{defU}). Using this simple observation we will often identify $\X$ with $\left ( \sum_{i=0}^{\infty} X_i\right)_{p}$. Our next theorem is similar to 
\cite[Theorem 16]{Dosev} and \cite[Theorem 4.6]{Apostol_lp}, but we state it and prove it in a more general setting and also include the case $p=\infty$.

\begin{theorem}\label{similaritythm}
Let $\X$ be a Banach space such that $\X\simeq\sumX$, 
$1\leq p\leq\infty$ or $p=0$. Let $T\in\opX$ be such that there exists a subspace $X\subset\X$ such that $X\simeq \X$, $T_{|X}$ is an isomorphism, $X+T(X)$ is complemented in $\X$ and $d(X, T(X))>0$.
Then there exists a decomposition $\decomp$ of $\X$ such that $T$ is similar to a matrix operator of the form
$$\left( \begin{array}{cc}
* & L  \\
* & * 
\end{array} \right)
$$
on $\X\oplus\X$, where $L$ is the left shift associated with $\decomp$.
\end{theorem}
\begin{proof}
Clearly $\X =  X\oplus T(X)\oplus Z$ where $Z$ is complemented in $\X$.
Note that without loss of generality we can assume that $Z$ is isomorphic to $\X$. Indeed, if this is not the case, let $X = X_1\oplus X_2$, $X\simeq X_1\simeq X_2$ and $X_1, X_2$ complemented in $X$ (hence also complemented in $\X$). Then $d(X_1, T(X_1))>0$ and $\X = X_1\oplus T(X_1)\oplus Z_1$ where $Z_1$ is a complemented subspace of $\X$, which contains the subspace $X_2\subset\X$, such that $X_2$ is isomorphic to $\X$ and complemented in $Z$. Applying  the Pe\l cz\'ynski decomposition technique (\cite[Proposition 4]{Pelcz_pojections}), we conclude that $Z_1$ is isomorphic to $X$. This observation plays an important role and will allow us to construct the decompositions we need during the rest of the proof.\\
Denote by $I-P$ the projection onto $T(X)$ with kernel $X+Z$.
Consider two decompositions $\displaystyle \decomp_{1} =\decompX $, $\displaystyle \decomp_{2} =\decompY $ of $\X$ such that $T(X) = Y_0 = X_1\oplus X_2\oplus\cdots$, $X_0 = Y_1\oplus Y_2\oplus\cdots$, $Y_1 = X$, and 
$Z = Y_2\oplus Y_3\oplus\cdots$.
Define a map $S$
$$
S\varphi = \LSo\varphi\oplus\LSt\varphi , \qquad \varphi\in\X
$$
 from $\X$ to $\X\oplus\X$. The map $S$ is invertible ($S^{-1}(a,b) = \RSo a + \RSt b$).
Just using the definition of $S$ and the formula for $S^{-1}$ we see that 

\begin{eqnarray*}
STS^{-1}(a,b) &=& ST(\RSo a + \RSt b) = S(T\RSo a + T\RSt b) \\
&=& (\LSo T\RSo a + \LSo T\RSt b)\oplus (\LSt T\RSo a + \LSt T\RSt b),
\end{eqnarray*}
hence
$$
STS^{-1} = \left( \begin{array}{cc}
* & \LSo T\RSt   \\
* & * 
\end{array} \right) .
$$
Let 
\begin{equation}\label{eq:A}
A=P_{Y_0}T\RSt = (I-P)T\RSt
\end{equation}
 and note that $A_{|P_{Y_0}\X} \equiv A_{|(I-P)\X} : (I-P)\X\to (I-P)\X$ is onto and invertible since $\RSt$ is an isomorphism on  $P_{Y_0}\X$ and $\RSt(P_{Y_0}\X) = Y_1 = X$.
Here we used the fact that $P_{Y_0}T$ is an isomorphism on $X$ ($PX = X$). Denote by $T_0 : (I-P)\X\to(I-P)\X$ the inverse of $A_{|P_{Y_0}\X}$ (note that $T_0$ is an automorphism on $(I-P)\X$) and consider $G : \X\to\X$ defined by
$$
G = I + T_0(I-P) - T_0A .
$$
We will show that $G^{-1} =  A+P$. In fact, from the definitions of $A$ and $T_0$
 it is clear that 
\begin{equation}\label{eq:AProp}
AT_0(I-P) = T_0A(I-P) = I-P \,,\, PT_0(I-P) = PA = 0 \,,\, (I-P)A = A
\end{equation}
and since $A$ maps onto $(I-P)\X$ and $AT_0 = I_{|(I-P)\X}$ we also have
\begin{equation}\label{eq:AProp1}
A - AT_0A = 0 .
\end{equation}
Now using (\ref{eq:AProp}) and (\ref{eq:AProp1}) we compute

\begin{eqnarray*}
(A+P)G &=& (A+P)(I + T_0(I-P) - T_0A) \\
&=& A + AT_0(I-P) - AT_0A + P = I-P + P = I\\
G(A+P) &=& (I + T_0(I-P) - T_0A)(A+P) \\
&=& A + P + T_0(I-P)A + T_0(I-P)P - T_0AA - T_0AP  \\
&=& A + P + T_0A - T_0AA- T_0AP \\
&=& P + (I - T_0A)A + T_0A(I-P)  \\
&=& P +(I-T_0A)(I-P)A +(I-P)\\
&=& I + ((I-P)-T_0A(I-P))A \\
&=& I + (I-P - (I-P))A = I .
\end{eqnarray*}
Using a similarity we obtain 
$$
\left( \begin{array}{cc}
I & 0   \\
0 & G^{-1} 
\end{array} \right)
\left( \begin{array}{cc}
* & \LSo T\RSt   \\
* & * 
\end{array} \right)
\left( \begin{array}{cc}
I & 0   \\
0 & G 
\end{array} \right) = 
\left( \begin{array}{cc}
* & \LSo T\RSt G   \\
* & * 
\end{array} \right) .
$$
It is clear that we will be done if we show that $\LSo=\LSo T\RSt G$. In order to do this consider
the equation $(A+P)G = I\Leftrightarrow AG + PG = I$. Multiplying both sides of the last equation on the left by $\LSo$ 
gives us $\LSo AG + \LSo PG = \LSo$. Using $\LSo P \equiv \LSo P_{X_0} = 0$ we obtain $\LSo AG = \LSo$. Finally, substituting $A$ from (\ref{eq:A})
 in the last equation yields
$$
\LSo = \LSo AG = \LSo P_{Y_0}T\RSt G = \LSo (I-P_{X_0})T\RSt G = \LSo T\RSt G,
$$
which finishes the proof.
\end{proof}

The following theorem was proved in \cite{Apostol_lp} for $X = \ell_p$, $1<p<\infty$, but inessential modifications give the result in these general settings.

\begin{theorem}\label{diagcomm}
Let $\X$ be a Banach space such that $\X\simeq\sumX$. Let $\decomp$ be a decomposition of $\X$ and let $L$ be the left shift associated with it. Then the matrix operator 
$$
\left( \begin{array}{cc}
T_1 & L   \\
T_2 & T_3 
\end{array} \right)
$$
acting on $\X\oplus\X$ is a commutator.
\end{theorem}
\begin{proof}
Let $\decomp =\{X_i\}$ be the given decomposition. Consider a decomposition $\decomp_1 = \{Y_i\}$ such that 
$\displaystyle Y_0 = \bigoplus_{i=1}^{\infty}X_i$ and $\displaystyle X_0 = \bigoplus_{i=1}^{\infty}Y_i$.
Now  there exists an operator $G$ such that 
$D_{L_{\decomp}}G = \RSo\LSo (T_1+T_3)$. This can be done using Lemma \ref{wstarconvlemma}, since 
$\RSo\LSo = I - P_{Y_0} = P_{X_0}$. By making the similarity
$$
\widetilde{T} := \left( \begin{array}{cc}
I & 0   \\
G & I 
\end{array} \right)
\left( \begin{array}{cc}
T_1 & L   \\
T_2 & T_3 
\end{array} \right)
\left( \begin{array}{cc}
I & 0   \\
-G & I 
\end{array} \right)
=
\left( \begin{array}{cc}
T_1 - LG & L   \\
* & T_3 + GL 
\end{array} \right)
$$
 we have $T_1 + T_3 - LG + GL = T_1 + T_3 - D_LG = T_1 + T_3 - \RSo\LSo (T_1+T_3) = P_{Y_0}(T_1+T_3)$.
Using Corollary \ref{wstarconvlemma}  again we deduce that $T_1 + T_3 - LG + GL $ is a commutator. Thus  by replacing $T$ by $\widetilde{T}$ 
we can assume that $T_1+T_3$ is a commutator, say $T_1+T_3= AB-BA$ and $\displaystyle \|A\|<1/2$ (this can be done by scaling).
Denote by $M_T$ left multiplication by the operator $T$. Then 
$\displaystyle \|M_RD_A\| < 1$ where $R$ is the right shift associated with $\decomp$. The operator
$T_0 = (M_I - M_RD_A)^{-1}M_R(T_3B - T_2)$ is well defined and it is easy to see that 
$$
\left( \begin{array}{cc}
A & 0   \\
T_3 & A-L 
\end{array} \right)
\left( \begin{array}{cc}
B & I   \\
T_0 & 0 
\end{array} \right)
-
\left( \begin{array}{cc}
B & I   \\
T_0 & 0 
\end{array} \right)
\left( \begin{array}{cc}
A & 0   \\
T_3 & A-L 
\end{array} \right)
=
\left( \begin{array}{cc}
T_1 & L   \\
T_2 & T_3 
\end{array} \right) .
$$
This finishes the proof.
\end{proof}

\section{Operators on $\linf$} 

\begin{definition}\label{def:les}
The left essential spectrum  of $T\in\opX$ is the set (\cite{Apostol_les} Def 1.1)
$$
\les = \{\lambda\in \CN : \inf_{x\in S_Y}\|(\lambda - T)x\| = 0\,\,\textrm{for all}\,\, Y\subset \X\,\,\textrm{s.t.}\,\, {\textrm{codim}}(Y)<\infty\}.
$$
\end{definition}
Apostol \cite[Theorem 1.4]{Apostol_les} proved that for any $T\in\opX$, $\les$ is a closed non-void set. The following lemma is a characterization of the operators not of the form
$\lambda I + K$ on the classical Banach sequence spaces. The proof presented here follows Apostol's ideas \cite[Lemma 4.1]{Apostol_lp}, but it is presented in a more general way.
\begin{lemma}\label{MovingLemma}
Let $\X$ be a Banach space isomorphic to $\ell_p$ for $1\leq p<\infty$ or $c_0$
and let $T\in\opX$.  Then the following are equivalent\\*
 \textbf{(1)} $T - \lambda I$ is not a compact operator for any $\LCN$.\\*
 \textbf{(2)} There exists an infinite dimensional complemented subspace $Y\subset\X$ such that $Y\simeq \X$, $T_{|Y}$ is an isomorphism and  $d(Y, T(Y))>0$.
\end{lemma}
\begin{proof} 
\textbf{(2)} $\Longrightarrow$ \textbf{(1)}\\*
Assume that $T = \lambda I + K$ for some $\LCN$ and some $K\in\comp{\X}$. Clearly 
$\lambda\neq 0$ since $T_{|Y}$ is an isomorphism. Now there exists a sequence $\seq{x}\subset S_Y$ such that $\|Kx_n\|\to 0$ as $n\to\infty$. Let $\displaystyle y_n = T\left ( \frac{x_n}{\lambda}\right )$ and note that
$$
\|x_n - y_n\|  = \left \|x_n - (\lambda I + K)\left (\frac{x_n}{\lambda}\right )\right \| = 
\left \|x_n - x_n - K\left( \frac{x_n}{\lambda}\right )\right \| = \frac{\|Kx_n\|}{\lambda}\to 0
$$
as $n\to\infty$ which contradicts the assumption $d(Y, T(Y))>0$. Thus $T - \lambda I$ is not a compact operator for any $\LCN$.\\*
\textbf{(1)} $\Longrightarrow$ \textbf{(2)}\\*
The proof in this direction follows the ideas of the proof of  Lemma 4.1 from \cite{Apostol_lp}.
Let $\lambda\in\les$. Then $T_1 = T - \lambda I$ is not a compact operator and 
$0\in\sigma_{l.e}(T_1)$. Using just the definition of the left essential spectrum, we find a normalized block basis sequence $\seq{x}$ of the standard unit vector basis of $\X$ such that  $\displaystyle \|T_1x_n\|<\frac{1}{2^n}$ for $n = 1,2,\ldots$. Thus if we denote
$Z = \overline{\textrm{span}}\{x_i\,:\, i = 1,2,\ldots\}$ we have $Z\simeq\X$ and 
${T_1}_{|Z}$ is a compact operator. Let $I-P$ be a bounded projection from $\X$ onto $Z$ (\cite[Lemma 1]{Pelcz_pojections}) so that $T_1(I-P)$ is compact. Now consider the operator $T_2 = (I-P)T_1P$. We have two possibilities:\\*
\textbf{Case I.} Assume that $T_2 = (I-P)T_1P$ is not a compact operator.
Then there exists an infinite dimensional subspace $Y_1\subset P\X$ on which $T_2$ is an isomorphism and hence using \cite[Lemma 2]{Pelcz_pojections} if necessary, we find a complemented subspace 
$Y\subset P\X$, such that $T_2$ is an isomorphism on $Y$. By the construction of the operator $T_2$ we immediately have $d(Y,(I-P)T_1P(Y))>0$ and hence $d(Y,T_1(Y))>0$.
Note that since $\X$ is prime and $Y$ is complemented in $\X$, $Y\simeq\X$ is automatic. Now we are in position to use Proposition \ref{distanceprop} to conclude that $d(Y,T(Y))>0$.\\*
\textbf{Case II.} Now we can assume that the operator $(I-P)T_1P$ is compact. Since
$T_1(I-P)$ is compact and using
$$
T_1 = T_1(I-P) + (I-P)T_1P + PT_1P
$$
we conclude that the operator $PT_1P$ is not compact. 
Using $\X \equiv P\X\oplus(I-P)\X$, we identify $P\X\oplus(I-P)\X$ with $\X\oplus\X$ via an isomorphism $U$, such that $U$ maps $P\X$ onto the first copy of $\X$ in the sum $\X\oplus\X$. Without loss of generality we assume that $T_1 = \bigl( \begin{smallmatrix}
T_{11}&T_{12}\\ T_{21}&T_{22}
\end{smallmatrix} \bigr)$ is acting on  $\X\oplus\X$. Denote by $P = \bigl( \begin{smallmatrix}
I&0\\ 0&0
\end{smallmatrix} \bigr)$ the projection from $\X\oplus\X$ onto the first copy of $\X$. In the new settings, we have that $T_{11}$ is not compact and $T_{21}, T_{22}$ and $T_{12}$ are compact operators.
Define  the operator $S$ on $\X\oplus\X$ in the following way:
$$
\sqrt{2}S = \left( \begin{array}{cr}
I & I  \\
I & -I 
\end{array} \right).
$$
Clearly $S^2 = I$ hence $S = S^{-1}$. Now consider the operator
$2(I-P)S^{-1}T_1SP$. A simple calculation shows that 
$$
2(I-P)S^{-1}T_1SP = \left( \begin{array}{cc}
0 & 0  \\
T_{11}+ T_{12} -  T_{21} - T_{22}& 0 
\end{array} \right)
$$
hence $(I-P)S^{-1}T_1SP$ is not compact.
Now we can continue as in the previous case to conclude that there exists a complemented subspace $Y\subset\X$ in the first copy of $\X\oplus\X$ for which  $d(Y,S^{-1}T_1S(Y))>0$ and hence  $d(SY,T_1(SY))>0$.
Again using Proposition \ref{distanceprop}, we  conclude that $d(SY,T(SY))>0$.
\end{proof}
\begin{remark}
We should note that the two conditions in the preceding lemma are equivalent to a third one, which is the same as \textbf{(2)} plus the additional condition that $Y\oplus T(Y)$ is complemented in $\X$. This is essentially  what was used for proving the complete classification of the commutators on $\ell_1$ in \cite{Dosev}, and 
$\ell_p$, $1<p<\infty$, and $c_0$ in \cite{Apostol_lp} and \cite{Apostol_c0}. The last mentioned condition will also play an important role in the proof of the complete classification of the commutators on $\linf$, but we should point out that once we have an infinite dimensional subspace $Y\subset\linf$ such that $Y\simeq \linf$, $T_{|Y}$ is an isomorphism and  $d(Y, T(Y))>0$, then $Y$ and $Y\oplus T(Y)$ will be automatically complemented in 
$\linf$.
\end{remark}

\begin{lemma}\label{basiclemma}
Let $T\in\oplinf$ and denote by $I$ the identity operator on $\linf$. Then the following are equivalent\\
(a) For each subspace $X\subset\linf, X\simeq c_0 $, there exists a constant $\lambda_X$ and a compact operator $K_X : X\to\linf$ depending on $X$ 
such that  $T_{|X} = \Trest{X}$.\\
(b) There exists a constant $\lambda$ such that $T = \lambda I + S$, where $S\in\sso$ .
\end{lemma}
\begin{proof}
Clearly (b) implies (a), since every strictly singular operator from $c_0$ to any Banach space is compact (\cite[Theorem 2.4.10]{Kalton}).
For proving the other direction we will first show that for every two subspaces $X,Y$ such that $X\simeq Y\simeq c_0$ we have $\lambda_X = \lambda_Y$. We have several cases.

{\bf{Case I.}} $X\cap Y = \{0\} ,\,\, d(X,Y)>0$.\\
Let $\seq{x}$ and $\seq{y}$ be bases for $X$ and $Y$, respectively, which are equivalent
to the usual unit vector basis of $c_0$. Consider the sequence $\seq{z}$ such that 
$z_{2i} = x_i\,,\, z_{2i-1} = y_i$ for $i=1,2,\ldots$. If we denote $Z = \overline{\textrm{span}}\{z_i\,:\, i = 1,2,\ldots\}$, then clearly $Z\simeq c_0$,
and, using the assumption of the lemma, we have that $T_{|Z} = \Trest{Z}$. Now using $X\subset Z$ we have that 
$\Trest{X}=(\Trest{Z})_{|X}$, hence  
$$
(\lambda_X - \lambda_Z)I_{|X} = (K_Z)_{|X} - K_X.
$$
The last equation is only possible if $\lambda_X = \lambda_Z$ since the identity is never a compact operator on a infinite dimensional subspace.
Similarly $\lambda_Y = \lambda_Z$ and hence $\lambda_X = \lambda_Y$.

{\bf{Case II.}} $X\cap Y = \{0\} ,\,\, d(X,Y)=0$.\\
Again let $\seq{x}$ and $\seq{y}$ be bases of $X$ and $Y$, respectively, which are equivalent
to the usual unit vector basis of $c_0$ and assume also that $\lambda_X\neq\lambda_Y$.
There exists a normalized block basis $\seq{u}$ of $\seq{x}$ and a normalized block basis $\seq{v}$ of $\seq{y}$ 
such that $\displaystyle \|u_i - v_i\|< \frac{1}{i}$.
Then $\|u_i - v_i\|\to 0\Rightarrow \|Tu_i - Tv_i\|\to 0\Rightarrow \|\lambda_Xu_i + K_Xu_i - \lambda_Yv_i - K_Yv_i\|\to 0$.
Since $u_i\to 0$ weakly (as a bounded block basis of the standard unit vector basis of $c_0$) we have $\|K_Xu_i\|\to 0$ and using $\|u_i - v_i\|\to 0$
we conclude that 
$$ 
\|(\lambda_X - \lambda_Y)v_i - K_Yv_i\|\to 0.
$$
Then there exists $N\in \N$ such that $\|K_Yv_i\|>\frac{|\lambda_X - \lambda_Y|}{2}\|v_i\|$ for $i>N$, which is impossible because $K_Y$ is a compact operator. 
Thus, in this case we also have $\lambda_X = \lambda_Y$.

{\bf{Case III.}} $X\cap Y = Z\neq \{0\} ,\,\, \dim (Z)=\infty$.\\
In this case we have $(\Trest{X})_{|Z}=(\Trest{Y})_{|Z}$ and, as in the first case, we rewrite the preceding equation in the form
$$
(\lambda_{X}I_{|X} - \lambda_{Y}I_{|Y})_{|Z} = (K_Y - K_X)_{|Z}.
$$
Again, as in {\bf Case I}, the last equation is only possible if $\lambda_X = \lambda_Y$ since the identity is never a compact operator on a infinite dimensional subspace.

{\bf{Case IV.}} $X\cap Y = Z\neq \{0\} ,\,\, \dim (Z)<\infty$.\\
Let $X = Z\bigoplus X_1$ and $Y = Z\bigoplus Y_1$. 
Then $X_1\cap Y_1 = \{0\}, \,X_1\simeq Y_1\simeq c_0$ and we can reduce to one of the previous cases. 

Let us denote $S = T - \lambda I$ where $\lambda = \lambda_X$ for arbitrary $X\subset\linf,\, X\simeq c_0$.
If  $S$ is not a strictly singular operator, then there is a subspace $Z\subset \linf,\, Z\simeq\linf$ such that
$S_{|Z}$ is an isomorphism (\cite[Corollary 1.4]{Rosenthal_linfty}), hence we can find $Z_1\subset Z\subset\linf,\,Z_1\simeq c_0$, such that $S_{|Z_1}$ is an isomorphism.
This contradicts the assumption that $S_{|Z_1}$ is a compact operator.
\end{proof}

The following corollary is an immediate consequence of Lemma \ref{basiclemma}.

\begin{corollary}\label{corr1}
Suppose $T\in\oplinf$ is such that $T-\lambda I\notin \sso$ for any $\LCN$. Then there exist a subspace $X\subset\linf,\, X\simeq c_0$ such that 
$(T-\lambda I)_{|X}$ is not a compact operator for any $\LCN$.
\end{corollary}

\begin{theorem}\label{c0theorem}
Let $T\in\oplinf$ be such that $T-\lambda I\notin \sso$ for any $\lambda$. Then there exists a subspace
$X\subset\linf$ such that $X\simeq c_0$, $T_{|X}$ is an isomorphism and $d(X, T(X))>0$.
\end{theorem}
\begin{proof}
By Corollary \ref{corr1} we have a subspace $X\subset\linf,\, X\simeq c_0$ such that $(T-\lambda I)_{|X}$ is not a compact operator for any $\lambda$.
Let $Z =  \overline{X\oplus T(X)}$ and let $P$ be a projection from $Z$ onto $X$ (such exists since $Z$ is separable and $X\simeq c_0$). We have two cases:\\*
\textbf{Case I.} The operator $T_1 = (I-P)TP$ is not compact.  Since $T_1$ is a non-compact operator from $X\simeq c_0$ into a Banach space we have that $T_1$ is an isomorphism on some  subspace $Y\subset X$, $Y\simeq c_0$ (\cite[Theorem 2.4.10]{Kalton}). Clearly, from the form of the operator $T_1$  we have  $d(Y, T_1(Y)) = d(Y, (I-P)TP(Y))>0$ and hence 
$d(Y, T(Y))>0$.\\*
\textbf{Case II.} If $(I-P)TP$ is compact and $\LCN$, then 
$(I-P)TP + PTP - \lambda I_{|Z} = TP- \lambda I_{|Z}$ is not compact and hence $PTP- \lambda I_{|Z}$ is not compact. Now for $T_2:=PTP\colon X\to X$  we apply Lemma \ref{MovingLemma} to conclude that there 
exists a subspace $Y\subseteq X$, $Y\simeq c_0$ such that 
$d(Y, PT(Y)) = d(Y, PTP(Y))>0$ and hence $d(Y, T(Y))>0$.
\end{proof}
The following theorem is an analog of Lemma \ref{MovingLemma} for the space $\linf$.
\begin{theorem}\label{maintheorem}
Let $T\in\oplinf$ be such that $T-\lambda I\notin \sso$ for any $\LCN$. Then there exists a subspace
$X\subset\linf$ such that $X\simeq \linf$, $T_{|X}$ is an isomorphism and $d(X, T(X))>0$.
\end{theorem}
\begin{proof}
 From Theorem \ref{c0theorem} we have a subspace
$Y\subset\linf$, $Y\simeq c_0$ such that $T_{|Y}$ is an isomorphism and $d(Y, T(Y))>0$. Let $N_k = \{3i+k\colon i = 0,1,\ldots \}$ for $k = 1,2,3$. 
There exists an isomorphism $\overline{S} \colon Y\oplus TY\to c_0(N_1)\oplus c_0(N_2)$ such that $\overline{S} (Y) = c_0(N_1)$ and 
$\overline{S}(TY) = c_0(N_2)$. Note that the space $Y\oplus TY$ is indeed a closed subspace of $\linf$ due to the fact that $d(Y, T(Y))>0$.
Now we use \cite[Theorem 3]{LR_automorphisms} to extend $\overline{S}$ to an automorphism $S$ on $\linf$.
Let $T_1 = STS^{-1}$ and consider the operator $(P_{N_2}T_1)_{|\linf (N_1)} \colon\linf (N_1)\to \linf (N_2)$, where $P_{N_2}$ is the natural 
projection onto $\linf (N_2)$. Since $T_1(c_0(N_1)) = c_0(N_2)$, by \cite[Proposition 1.2]{Rosenthal_linfty} there exists an infinite set  $M\subset N_1$  such that $(P_{N_2}T_1)_{|\linf (M)}$ is an isomorphism. This immediately yields
\begin{equation*}
d(\linf (M), P_{N_2}T_1(\linf (M)))>0
\end{equation*}
and hence
\begin{equation}\label{eql1}
d(\linf (M), T_1(\linf (M)))>0.
\end{equation}
Finally, recall that $T_1 = STS^{-1}$, thus
$$
d(\linf (M), STS^{-1}(\linf (M)))>0
$$
and hence $d(S^{-1}(\linf (M)), TS^{-1}(\linf (M)))>0$.
\end{proof}
\noindent
Finally, we can prove our main result.
\begin{theorem}\label{mainthm}
An operator $T\in\oplinf$ is a commutator if and only if $\,\,T-\lambda I\notin \sso$ for any $\lambda\neq 0$.
\end{theorem}
\begin{proof}
Note first that if $T$ is a commutator, from the remarks we made in the introduction it follows that $T-\lambda I$ cannot be strictly singular for any 
$\lambda\neq 0$. For proving the other direction we have to consider two cases:\\
{\bf{Case I.}} If $T\in\sso$ ($\lambda = 0$), the statement of the theorem follows from \cite[Theorem 23]{Dosev}.\\
{\bf{Case II.}} If $\,\,T-\lambda I\notin\sso$ for any $\LCN$, then we apply Theorem \ref{maintheorem} 
to get $X\subset\linf$ such that $X\simeq\linf$, $T_{|X}$ an isomorphism and $d(X,TX)>0$. The subspace
$X+TX$ is isomorphic to $\linf$ and thus is complemented in $\linf$.
Theorem \ref{similaritythm} now yields
that $T$ is similar to an operator of the form $\left( \begin{array}{cc}
* & L  \\
* & * 
\end{array} \right)
$.
Finally, we apply Theorem \ref{diagcomm} to complete the proof.
\end{proof}

\section{Remarks and problems}
We end this note with some comments and questions that arise from our work.\\
First consider the set
$$
\IX = \{T\in\opX\,:\, I_{\X}\,\, \textrm{does not factor through}\,\, T \}.
$$

This set comes naturally from our investigation of the commutators on $\ell_p$ for $1\leq p\leq\infty$. We know (\cite[Theorem 18]{Dosev}, \cite[Theorem 4.8]{Apostol_lp}, \cite[Theorem 2.6]{Apostol_c0}) that the non-commutators on $\ell_p$, $1\leq p<\infty$ and $c_0$ have the form $\lambda I + K$ where $K\in\IX$ and $\lambda\neq 0$, where $\IX=\comp{\ell_p}$ is actually the largest ideal in $\mathcal{L}(\ell_p)$ (\cite{GM}), and, in  this paper we showed (Theorem \ref{mainthm}) that the non-commutators on $\linf$ have the form $\lambda I + S$ where $S\in\IX$ and $\lambda\neq 0$, where $\IX=\sso$. Thus, it is natural to ask the question for which Banach spaces $\X$ is the set $\IX$  the largest ideal in  $\opX$? 
 Let us also mention that in addition to the already mentioned spaces, if $\X = L_p(0,1)$, 
$1\leq p<\infty$,  then $\IX$ is again the largest ideal in $\opX$ (cf. \cite{Enflo_Starbird} for the case $p=1$ and  \cite[Proposition 9.11]{JMST} for $p>1$).

First note that the set $\IX$ is closed under left and right multiplication with operators from $\opX$, so the question whether $\IX$ is an ideal is equivalent to the question whether $\IX$ is closed under addition. Note also that if $\IX$ is an ideal then it is automatically the largest ideal in $\opX$ and hence closed, so the question we will consider is under what conditions we have
\begin{equation}\label{idealprop}
\IX + \IX\subseteq \IX.
\end{equation}
The following proposition  gives a sufficient condition for (\ref{idealprop}) to hold.
\begin{proposition}
Let $\X$ be a Banach space such that for every $T\in\opX$ we have $T\notin\IX$ or $I-T\notin\IX$. Then $\IX$ is the largest (hence closed) ideal in $\opX$.
\end{proposition}
\begin{proof}
Let $S,T\in\IX$ and assume that $S+T\notin\IX$. By our assumption,  there exist two operators $U\colon\X\to\X$ and $V\colon \X\to\X$ which make the following diagram commute:
$$
\begin{diagram}[height=2em,width=4.5em,abut] 
\X& \rTo^{S+T} & \X\\
{U\,}\uTo &  &\dTo > {V} \\
\X& \rTo^{I}& \X
\end{diagram}
$$
Denote $W = (S+T)U(\X)$ and let $P\colon\X\to W$ be a projection onto $W$ (we can take $P = (S+T)UV$). Clearly   $VP(S+T)U = I$. Now $S, T\in\IX$ implies $VPSU, VPST\in\IX$ which is a contradiction since $VPSU + VPTU = I$.
\end{proof}

Let us just mention that the conditions of the proposition above are satisfied for $\X = C([0,1])$ 
(\cite[Proposition 2.1]{LP}) hence $\IX$ is the largest ideal in $\mathcal{L}(C([0,1]))$ as well.

We should point out that there are Banach spaces for which $\IX$ is not an ideal in $\opX$. 
In the space $\ell_p\oplus\ell_q$, $1\leq p < q<\infty$, there are exactly two maximal ideals (\cite{Porta}), namely, the closure of the ideal of the operators that factor through $\ell_p$, which we will denote by $\alpha_p$, and  the closure of the ideal of the operators that factor through $\ell_q$, which we will denote by $\alpha_q$. In this particular space, the first author proved a necessary and sufficient condition for an operator to be a commutator:

\begin{theorem}
(\cite[Theorem 20]{Dosev}) Let $P_{\ell_p}$ and $P_{\ell_q}$ be the natural projections
from $\ell_p\oplus\ell_q$ onto $\ell_p$ and $\ell_q$, respectively. Then $T$ is a commutator if and only of $P_{\ell_p}TP_{\ell_p}$ and $P_{\ell_q}TP_{\ell_q}$ are commutators as operators acting on $\ell_p$ and $\ell_q$ respectively.
\end{theorem}
If we denote $T = \bigl( \begin{smallmatrix}
T_{11}&T_{12}\\ T_{21}&T_{22}
\end{smallmatrix} \bigr)$, the last theorem implies that $T$ is not a commutator if and only if $T_{11}$ or $T_{22}$ is not a commutator as an operator acting on $\ell_p$ or $\ell_q$ respectively. Now using the classification of the commutators on $\ell_p$ for $1\leq p<\infty$ and the results in \cite{Porta}, it is easy to deduce that an operator on $\ell_p\oplus\ell_q$ is not a commutator if and only if it has the form $\lambda I + K$ where
$\lambda\neq 0$ and $K\in\alpha_p\cup\alpha_q$. We can generalize this fact, but first we need a definition and a Lemma that follows easily from \cite[Corollary 21]{Dosev}.

\textbf{Property $\propP$}. We say that a Banach space $\X$ has property $\propP$ if $T\in\opX$ is not a commutator if and only if $T = \lambda I + S$, where $\lambda\neq 0$ and $S$ belongs to some proper ideal of $\opX$.

All the Banach spaces we have considered so far have property $\propP$ and our goal now is to show that property $\propP$ is closed under taking finite sums under certain conditions imposed on the elements of the sum.

\begin{lemma}\label{remlemma1}
Let $\{X_i\}_{i=1}^n$ be a finite sequence of Banach spaces that have property $\propP$. Assume also that all operators $A\colon X_i\to X_i$ that factor through $X_j$ are in the intersection of all maximal ideals in $\mathcal{L}(X_i)$ for each $i,j = 1,2,\ldots ,n,\,\, i\neq j$.
Let $\X = X_1\oplus X_2\oplus\cdots\oplus X_n$  and let $P_i$  be  the natural projections from $\X$ onto $X_i$ for $i = 1,2,\ldots ,n$. Then $T\in\opX$ is a commutator if and only if for each $1\leq i\leq n$, $P_iTP_i$  is a commutator as an operator acting on $X_i$. 

\end{lemma}
\begin{proof}
The proof is by induction and it mimics the proof of \cite[Corollary 21]{Dosev}.
First consider the case $n=2$.
Let 
$T = \left( \begin{array}{cc}
A & B   \\
C & D 
\end{array} \right)
$ where $A:X_1\to X_1, D: X_2\to X_2, B:X_2 \to X_1, C:X_1\to X_2$.
If $T$ is a commutator, then  $T = [T_1,T_2]$ for some $T_1, T_2\in\opX$. Write 
$T_i = \left( \begin{array}{cc}
A_i & B_i   \\
C_i & D_i 
\end{array} \right)
$ for $i = 1,2$. A simple computation shows that  
$$
T = \left( \begin{array}{cc}
[A_1, A_2] + B_1C_2 - B_2C_1 & A_1B_2 + B_1D_2 - A_2B_1 - B_2D_1   \\
C_1A_2 + D_1C_2 - C_2A_1 - D_2C_1 & [D_1, D_2] + C_1B_2 - C_2B_1 
\end{array} \right) .
$$
From the fact that $X_1$ and $X_2$ have property $\propP$, and the fact that the $B_1C_2,\,B_2C_1$ lie in the intersection of all maximal ideals in $\mathcal{L}(X_1)$ and $C_1B_2,\,C_2B_1$ lie in the intersection of all maximal ideals in $\mathcal{L}(X_2)$ we immediately deduce that
the diagonal entries in the last representation of $T$ are commutators. In the preceding argument we used the fact that a perturbation of a commutator on a Banach space $X$ having property $\propP$ by an operator that lies in the intersection of all maximal ideals in $\mathcal{L}(X)$ is still a commutator. To show this fact assume that 
$A\in \mathcal{L}(X)$ is a commutator,  $B\in \mathcal{L}(X)$ lies in the intersection of all maximal ideals in $\mathcal{L}(X)$ and $A+B = \lambda I + S$ where $S$ is an element of some ideal $M$ in $\mathcal{L}(X)$. Now using the simple observation that every ideal is contained in some maximal ideal, we conclude that $S-B$ is contained in a maximal ideal, say $\tilde{M}$ containing $M$ hence $A - \lambda I\in\tilde{M}$, which is a contradiction with the assumption that $X$ has property $\propP$.
 
For the other direction we apply  \cite[Lemma 19]{Dosev} which concludes the proof in the case $n=2$.
The general case follows from the same considerations as in the case $n=2$ in a obvious way.
\end{proof}

Our last corollary shows that property $\propP$ is preserved under taking finite sums of Banach spaces having property $\propP$ and some additional assumptions as in Lemma \ref{remlemma1}. 

\begin{corollary}
Let $\{X_i\}_{i=1}^n$ be a finite sequence of Banach spaces that have property $\propP$. Assume also that all operators $A\colon X_i\to X_i$ that factor through $X_j$ are in the intersection of all maximal ideals in $\mathcal{L}(X_i)$ for each $i,j = 1,2,\ldots ,n,\,\, i\neq j$.
Then $\X = X_1\oplus X_2\oplus\cdots\oplus X_n$ has property $\propP$.
\end{corollary}
\begin{proof}
Assume that $T\in\opX$ is not a commutator. Using Lemma \ref{remlemma1}, this can happen if and only if $P_iTP_i$ is not commutator on $X_i$ for some $i\in \{1,2,\ldots ,n\}$ and without loss of generality assume that $i=1$.
Since $P_1TP_1$ is not a commutator and $X_1$ has property $\propP$ then $P_1TP_1 = \lambda I_{X_1} + S$ where $S$ belongs to some maximal ideal $J$ of $\mathcal{L}(X_1)$. Consider 

\begin{equation}
M = \{T\in\opX\,:\, P_1TP_1 \in J \}.
\end{equation}

Clearly, if $T\in M$ and $A\in\opX$, then $AT, TA\in M$ because of the assumption on the operators from $X_1$ to $X_1$ that factor through $X_j$. It is also obvious that $M$ is closed under addition, hence $M$ is an ideal.
Now it is easy to see that $T-\lambda I\in M$ which shows that all non-commutators have the form  $\lambda I + S$, where $\lambda\neq 0$ and $S$ belongs to some proper ideal of $\opX$.

The other direction follows from our comment in the beginning of the introduction that no operator of the form
$\lambda I + S$ can be a commutator for any $\lambda \neq 0$ and any operator $S$ which lies in a proper ideal of $\opX$.

\end{proof}

\noindent
   Detelin Dosev\\
   Department of Mathematics\\
   Texas A\&M University\\
   College Station, Texas 77843\\
   USA\\
   \textbf{dossev@math.tamu.edu}

\noindent
   William B. Johnson\\
   Department of Mathematics\\
   Texas A\&M University\\
   College Station, Texas 77843\\
   USA\\
   \textbf{johnson@math.tamu.edu}
\end{document}